\documentclass[11pt]{amsart}

\usepackage[T1]{fontenc}
\usepackage[utf8]{inputenc}
\usepackage{lmodern}
\usepackage{amsmath,amssymb,amsthm,mathtools,mathrsfs,amsrefs, bbm}
\usepackage[margin=1.25in]{geometry}
\usepackage[shortlabels]{enumitem}
\usepackage[colorlinks=true,linkcolor=blue,citecolor=blue,urlcolor=blue]{hyperref}
\usepackage{microtype}
\usepackage{comment}

\usepackage{tikz}
\usetikzlibrary{arrows.meta,calc,positioning}

\newtheorem{theorem}{Theorem}[section]
\newtheorem{lemma}[theorem]{Lemma}
\newtheorem{proposition}[theorem]{Proposition}
\newtheorem{corollary}[theorem]{Corollary}
\theoremstyle{definition}
\newtheorem{remark}[theorem]{Remark}
\newtheorem{definition}[theorem]{Definition}

\newcommand{\Cu}{\operatorname{Cu}}
\newcommand{\cS}{\mathcal{S}}
\newcommand{\cT}{\mathcal{T}}
\newcommand{\C}{\mathbb{C}}
\newcommand{\N}{\mathbb{N}}
\newcommand{\cU}{\mathcal{U}}

\renewcommand{\epsilon}{\varepsilon}

\title{Regularity and selfless cp maps}
\author{Miles Gould, Leonel Robert}

\begin{document}

\begin{abstract}
We investigate the regularity properties of C*-algebras admitting a weakly selfless unital completely positive map.
\end{abstract}
\maketitle

\section{Introduction}
A pair $(A,\rho)$ consisting of a unital C*-algebra $A$ and a
GNS-faithful state $\rho$ is called selfless if, 
for some unital C*-algebra
$C\neq\C$ with a GNS-faithful state $\kappa$ and some ultrapower $(A^{\cU},\rho^{\cU})$, there exists
a state-preserving *-homomorphism
$$
(A,\rho)*(C,\kappa)\longrightarrow (A^{\cU},\rho^{\cU})
$$
whose restriction to $A$ is the diagonal embedding.
If $\rho$ is not a trace, then $(A,\rho)$ is selfless if and only if $A$
is simple and purely infinite, while if $\rho$ is a trace, then $(A,\rho)$
is selfless if and only if $A$ is  infinite dimensional, simple, and monotracial with strict comparison relative to $\rho$
\cite{selfless,GO26,OZA,selfless2}.
In the tracial case, $A$ has stable rank one 
and its unique trace is also its unique 2-quasitrace \cite{selfless}.
Selflessness has found many applications in the study of regularity properties of C*-algebras, e.g., as a mechanism for verifying the strict comparison property for various classes of C*-algebras (see for example \cite{amrutametal, Bell, floresetal, hayesetal2,  OZA, vigdorovich}).

A broader notion of selflessness, relative to a 
C*-correspondence,
was introduced in \cite{GJKEPR}. Applied to the KSGNS C*-correspondence of  a completely positive map $\phi\colon A\to A$, it gives a notion of selflessness for the pair $(A,\phi)$.  Unlike  
its state counterpart, this setting accommodates non-simple C*-algebras and C*-algebras with more than one trace.
 It is natural to ask which regularity properties
persist in this setting. 
This question was partially explored in \cite{GJKEPR}.
We pursue it here more directly. More concretely,
we investigate which regularity properties of the underlying
C*-algebra $A$ follow from the existence of a weakly selfless
unital completely positive map $\phi\colon A\to A$.

Recall that a C*-algebra is called \emph{pure} if its Cuntz semigroup
is almost unperforated and almost divisible. 



\begin{theorem}\label{thm:main2}
Let $A$ be a unital C*-algebra and let $\phi\colon A\to A$
be a ucp map such that $(A,\phi)$ is weakly selfless.
Suppose that for every closed two-sided ideal $I\subseteq A$
and every approximate unit $(e_\lambda)_\lambda$ of $I$,
$(\phi(e_\lambda))_\lambda$ is also an approximate unit of $I$.
Then $A$ is pure and every lower semicontinuous
$[0,\infty]$-valued $2$-quasitrace on $A$ is a trace.
\end{theorem}

When $A$ is simple, the condition on approximate units is automatic
for a ucp map. Thus, Theorem~\ref{thm:main2} implies that if $A$
is simple and admits a weakly selfless ucp map, then it is pure
and its bounded $2$-quasitraces are traces.
If $A$ moreover admits a tracial state, then Lin's result
\cite[Corollary~4.5]{Lin} implies that $A$ has stable rank one.
Consequently, in the simple case the
purely infinite/stable-rank-one dichotomy is again present in the
setting of the preceding theorem, while
allowing more than one tracial state in the finite case. 



Preservation of pureness under tensor products 
was recently investigated by Seth and Vilalta 
in \cite{seth-vilalta}. 
  We can apply Theorem \ref{thm:main2} to
 the slice conditional expectations of a minimal tensor product
 to say something about this question. Indeed,
the conditional expectation
$\rho\otimes\operatorname{id}_B\colon A\otimes B\to 1\otimes B$
is selfless whenever $(A,\rho)$ is (scalar) selfless and $B$
is exact or whenever $(A,\rho)$ is completely selfless and $B$ is arbitrary
\cite[Theorem~5.4(iv) and Remark~5.5]{GJKEPR}.  Preservation
of approximate units of ideals is guaranteed once either $A$ or $B$
is exact.
 Thus, Theorem~\ref{thm:main2} implies: 


\begin{corollary}\label{cor:tensor}
Let $(A,\rho)$ be a scalar C*-probability space and let $B$
be a unital C*-algebra.
Suppose that either $(A,\rho)$ is selfless and $B$ is exact,
or $(A,\rho)$ is completely selfless and $A$ is exact.
Then the minimal tensor product $A\otimes B$ is pure and every lower semicontinuous $[0,\infty]$-valued $2$-quasitrace on $A\otimes B$
is a trace.
\end{corollary}

Choosing $A=C_r^*(\mathbb F_2)$,
which is both exact and completely selfless (the latter by \cite{OZA}),
we deduce from Corollary \ref{cor:tensor} that for any unital C*-algebra $B$ the minimal tensor product $C_r^*(\mathbb F_2)\otimes B$ is pure and its l.s.c.
$[0,\infty]$-valued 2-quasitraces are traces. Choosing $A=M_{2^\infty}$, which is selfless by \cite{selfless}, and $B$ unital and exact, we obtain
that every l.s.c. $[0,\infty]$-valued $2$-quasitrace
on $M_{2^\infty}\otimes B$ is a trace. Moreover, every lower semicontinuous $[0,\infty]$-valued
$2$-quasitrace on $B$ extends to $M_{2^\infty}\otimes B$  (by writing $M_{2^\infty}\otimes B=\varinjlim M_{2^n}(B)$). We thus obtain  a new proof of Haagerup's theorem \cite{haag} that  bounded 2-quasitraces on an exact C*-algebra are traces, and of
its extension by Blanchard and Kirchberg to lower semicontinuous,
$[0,\infty]$-valued $2$-quasitraces
\cite[Remark~2.29(i)]{blanchard-kirchberg}.

\begin{corollary}[Haagerup; Blanchard-Kirchberg]\label{cor:haagerup}
Let $B$ be an exact C*-algebra. Then every lower semicontinuous $[0,\infty]$-valued $2$-quasitrace on $B$ is a trace.
\end{corollary}


We work in two settings in parallel: assuming only the existence of a 
weakly selfless ucp map $\phi\colon A\to A$, we obtain regularity properties for full elements of the  Cuntz semigroup and we show that bounded 2-quasitraces on $A$ are traces.
In order to extend these results to non-full Cuntz elements and l.s.c. $[0,\infty]$-valued 2-quasitraces, we assume that the cp map $\phi$ preserves approximate units of ideals of $A$. The proof strategies are essentially the same in the two settings, but the case of full Cuntz elements is less technical. We treat it  first and use it as a template for the other case.

En route to proving  Theorem \ref{thm:main2}, we 
introduce a property which  we  term strict $\phi$-comparison. 
Strict $\phi$-comparison allows us to deduce Cuntz comparison
of positive elements from comparison of their generalized
$\phi$-ranks $d_\phi(\cdot)$
(see Definition~\ref{def:strict-phi-comparison}).
We show that if $(A,\phi)$ is weakly selfless and $\phi$ preserves
approximate units of ideals, then $(A,\phi)$ has strict
$\phi$-comparison.
In turn, strict $\phi$-comparison implies that the Cuntz semigroup of $A$
is almost unperforated and that $2$-quasitraces on $A$ are traces; equivalently, $A$ has strict comparison by traces. Weak selflessness and the preservation of approximate units of ideals by $\phi$ also imply that $A$ satisfies the global Glimm halving property. 
Finally, combining strict comparison by traces with the global Glimm
halving property, we obtain pureness of $A$ by Thiel's recent result
\cite[Theorem~4.4]{thiel-ranks}.

\begin{tikzpicture}[
  statement/.style={
    align=center,
    text width=3.9cm,
    minimum height=1.1cm,
    inner sep=4pt
  },
  implication/.style={
    double,
    double equal sign distance,
    -{Implies},
    shorten <=4pt,
    shorten >=4pt
  },
  every node/.style={font=\small}
]

\node[statement] (selfless)
  {Weakly selfless and\\preserving a.u.'s of ideals};
\node[statement, below=0.8cm of selfless] (comparison)
  {Strict $\phi$-comparison};
\node[statement, right=0.9cm of selfless] (glimm)
  {Global Glimm halving};
\node[statement] (unperforation) at (glimm |- comparison)
  {Almost unperforation\\and $2$-quasitraces are traces};

\draw[implication] (selfless.south) -- (comparison.north);
\draw[implication] (selfless.east) -- (glimm.west);
\draw[implication] (comparison.east) -- (unperforation.west);

\coordinate (midpoint) at
  ($(glimm.east)!0.5!(unperforation.east)$);
\node[inner sep=4pt] (joint) at
  ($(midpoint)+(0.65cm,0)$) {$\land$};
\draw[line width=0.45pt] (glimm.east) -| (joint.north);
\draw[line width=0.45pt] (unperforation.east) -| (joint.south);

\node[statement, right=0.8cm of joint] (pure)
  {Pure\\and $2$-quasitraces are traces};
\draw[implication] (joint.east) -- (pure.west);

\end{tikzpicture}

\emph{AI use statement}: ChatGPT 6.0 Pro was used interactively at various stages in developing the arguments and proofreading the paper. The authors take full responsibility for its contents.

L.R. acknowledges support from NSF grant DMS-2555075.

\section{Preliminaries}
\subsection{Selfless and weakly selfless cp maps}
Throughout, $A$ is a unital C*-algebra. 
Let $\phi\colon A\to A$ be a completely positive (cp) map.
Let $H_\phi=A\otimes_\phi A$ be its KSGNS C*-correspondence
(\cite{brown-ozawa}). Write $\xi_\phi\in H_\phi$ for the class of $1\otimes1$, and note that  $H_\phi=\overline{A\xi_\phi A}$. 

Let $\mathcal F_\phi$ be the Fock correspondence obtained from $H_\phi$:
$$
\mathcal F_\phi
=A\oplus\bigoplus_{n=1}^{\infty}H_\phi^{\otimes_A n}.
$$
Let $\ell_\phi\in B(\mathcal F_\phi)$ denote the left creation operator associated to $\xi_\phi$ and  set $s_\phi=\ell_\phi+\ell_\phi^*$.
The Toeplitz-Pimsner algebra and Shlyakhtenko's algebra
of the semicircular $s_\phi$ are, respectively,
$$
\cT_\phi=C^*(A,\ell_\phi),
\qquad
\cS_\phi=C^*(A,s_\phi).
$$
The vacuum expectation $E_\phi\colon\cT_\phi\to A$ is
given by $E_\phi(x)=\langle\Omega,x\Omega\rangle$, where
$\Omega=1_A$ is  the vacuum vector.  We note that
$E_\phi(s_\phi)=0$ and
$$
\ell_\phi^*a\ell_\phi = E_\phi(s_\phi a s_\phi)=\phi(a), \qquad a\in A.
$$

Following \cite{GJKEPR}, we call $(A,\phi)$ \emph{selfless}
if, for some ultrapower $A^{\cU}$, there exists a
*-homomorphism $\sigma\colon\cS_\phi\to A^{\cU}$
fixing $A$ and satisfying
$$
\phi^{\cU}\circ\sigma
=\sigma\circ\phi\circ E_\phi.
$$
We call $(A,\phi)$ \emph{weakly selfless} if there exists
a *-homomorphism $\widetilde\sigma\colon C\to A^{\cU}$ fixing $A$,
where
$$
C:=(*_A)_{i=1}^{\infty}\cS_\phi
$$
is the amalgamated free product reduced with respect to the
expectations $E_\phi$. We denote by $E\colon C\to A$
its canonical expectation onto $A$.
By \cite[Theorem~3.8]{GJKEPR}, if $(A,\phi)$ is selfless,
then such a map $\widetilde\sigma\colon C\to A^{\cU}$ can be chosen
to satisfy
$$
\phi^{\cU}\circ\widetilde\sigma
=\sigma\circ\phi\circ E.
$$
(Note that we do not assume this in the definition of weak selflessness.)
In particular, selflessness implies weak selflessness.

\begin{lemma}\label{lem:theta}
Let $\phi\colon A\to A$ be cp and (weakly) selfless and let
$b\in A$ be positive.
Then $\theta(a)=\phi(b\phi(a)b)$ is a (weakly) selfless cp map.
\end{lemma}

\begin{proof}
First assume $b$ is invertible. This case follows from permanence properties of (weak) selflessness. We deal with the case that $\phi$ is selfless, the weakly selfless case follows by  the same arguments.

By \cite[Theorem~5.4(i)]{GJKEPR}, the cp map
$$
\psi(x)=b^{\frac12}
\phi(b^{\frac12}xb^{\frac12})b^{\frac12}
$$
is selfless.
By \cite[Corollary~5.6]{GJKEPR}, $\psi^2$ is selfless,
and by \cite[Theorem~5.4(i)]{GJKEPR}, so is
$$
\theta(x)=
b^{-\frac12}
\psi^2(b^{-\frac12}xb^{-\frac12})b^{-\frac12}
=\phi(b\phi(x)b).
$$

Now let $b$ be positive. Then $b_n=b+\frac1n$ is invertible, so 
$$
\theta_n(x)=\phi(b_n\phi(x)b_n)
$$
is selfless and $\|\theta_n-\theta\|\rightarrow 0$, where $\theta(x)=\phi(b\phi(x)b)$. By \cite[Theorem 5.4 (ii)]{GJKEPR}, a point-norm limit
of weakly selfless maps is again weakly selfless. The proof 
can be adapted to also show that an operator norm limit of selfless maps
is selfless. We thus conclude that  $\theta$ is selfless.
\end{proof}

Let $\phi\colon A\to A$ be a cp map.
Set 
$$
C=(*_A)_{i=1}^{\infty}\cS_\phi, \quad \widetilde \phi=\phi\circ E,
$$
where $E\colon C\to A$ denotes  the canonical expectation. Regard
$\widetilde\phi=\phi\circ E$ as a cp map on $C$.

\begin{lemma}\label{lem:free-dilation}
$(C,\widetilde\phi)$ is selfless.
\end{lemma}

\begin{proof}
Set $D=C*_A\cS_\phi$, and let $F\colon D\to C$ be the
canonical expectation.
By \cite[Lemma~3.7]{GJKEPR}, we have
$\cS_{\widetilde\phi}\cong D$, with the isomorphism fixing $C$
and identifying the vacuum expectation with $F$.

For each $n$, reindexing factors in $D=((*_A)_{i=1}^{\infty}\cS_\phi)*_A\cS_\phi$
gives an
isomorphism $\rho_n\colon D\to C$ which fixes the first
$n$ factors $\cS_\phi$ in $C$ and preserves the expectations onto $A$.
Thus, $\rho_n(c)\to c$ for every $c\in C$, and
$E\circ\rho_n=E\circ F$.
Fix a free ultrafilter $\cU$ on $\N$ and define
$$
\sigma(d)=[(\rho_n(d))_n]\in C^{\cU}.
$$
Then $\sigma$ fixes $C$ and satisfies
$$
\widetilde\phi^{\cU}\circ\sigma
=\sigma\circ\widetilde\phi\circ F.
$$
Hence, $(C,\widetilde\phi)$ is selfless.
\end{proof}

\subsection{Pure C*-algebras}\label{subsec:pure-cstar-algebras}
Let $A$ be a C*-algebra and denote by $A_+$ the set of positive elements of $A$. Given $a,b\in A_+$  we write $a\precsim b$
to indicate that $a$ is Cuntz subequivalent to $b$, and $a\sim b$ to indicate that they are Cuntz equivalent. We denote by $\Cu(A)$ the Cuntz semigroup of $A$, i.e., the set of Cuntz equivalence classes of $(A\otimes\mathcal K)_+$. Below, we recall some definitions and facts around the Cuntz semigroup and regularity of C*-algebras. We refer the reader  to \cite{GardellaPerera} for background on the Cuntz semigroup.

The Cuntz semigroup $\Cu(A)$ is said to be \emph{almost unperforated} if 
$$
(k+1)[a]\leq k[b]\Longrightarrow [a]\leq [b],
$$
for all $[a],[b]\in \Cu(A)$ and $k\in \N$. If the above implication holds
whenever $[b]$ is a full element, i.e., $\infty [b]$ is the largest element in $\Cu(A)$, then we say that $\Cu(A)$ is almost unperforated on full elements. 

$\Cu(A)$ is said to be \emph{almost divisible} if for all $[a]\in \Cu(A)$,  $[a']\ll [a]$, and $N\in \N$, there exists $[b]\in \Cu(A)$ such that 
$$
N[b]\leq [a],\quad [a']\ll (N+1)[b].
$$
Here $\ll$ denotes the way-below relation in $\Cu(A)$ (\cite[Definition 4.1]{GardellaPerera}).

If $\Cu(A)$ is both almost unperforated and almost divisible then 
$A$ is said to be a \emph{pure} C*-algebra.

Almost unperforation can be reformulated as a strict comparison property
involving 2-quasitraces on $A$. Let $QT(A)$ denote the set of lower semicontinuous  $[0,\infty]$-valued 2-quasitraces on $A$. Each $\tau\in QT(A)$ gives rise to $d_\tau\colon \Cu(A)\to [0,\infty]$ given by
$$
d_\tau([a])=\lim_n \tau(a^{\frac1n}).
$$
Then  $\Cu(A)$ is almost unperforated if and only if for all $0\leq\delta<1$ and $[a],[b]\in \Cu(A)$,
$$
d_\tau([a])\leq \delta d_\tau([b])\text{ for all }\tau\in QT(A)
\Longrightarrow [a]\leq [b].
$$
See \cite{ERS}. If the same implication holds when the inequality 
$d_\tau([a])\leq \delta d_\tau([b])$
is required only for those $\tau\in QT(A)$ that are traces, then we say that $A$ has strict comparison by traces.
Strict comparison by traces implies that 2-quasitraces in $QT(A)$ are traces \cite[Theorem~3.6(i)]{NgRobert}. Consequently, we have the following:

\begin{proposition}\label{prop:comp-by-traces}
The C*-algebra $A$ has strict comparison by traces if and only if $\Cu(A)$ is almost unperforated and the 2-quasitraces in $QT(A)$ are traces. 
\end{proposition}

\subsection{Strict $\phi$-comparison}
Let $A$ be a unital C*-algebra and let $\phi\colon A\to A$ be a cp map.
Let $\phi^{**}\colon A^{**}\to A^{**}$ be its normal extension, 
and for $a\in A_+$ define
$$
d_\phi(a)=\phi^{**}(\chi_{(0,\|a\|]}(a)).
$$

\begin{remark}\label{rmk: nobidual-comparison}
Although the definition of $d_\phi$ references the bidual of $A$, we will not
need to work with it directly, since the comparisons involving
$d_\phi$ that we use below admit the following characterization:
for $a,b\in A_+$,
$$
d_\phi(a)\leq b
\Longleftrightarrow
\phi(h)\leq b
\quad\text{for every positive contraction }h\in C^*(a).
$$
\end{remark}

Given a cp map $\phi\colon A\to A$, denote by
$\phi^{(k)}\colon M_k(A)\to A$ the normalized cp map
$$
\phi^{(k)}(a)=\frac1k\sum_{i=1}^k\phi(a_{ii}).
$$

\begin{definition}\label{def:strict-phi-comparison}
We say that $(A,\phi)$ has strict $\phi$-comparison if for all
$k\in\N$, positive elements $a,b\in M_k(A)$, $\epsilon>0$,
and $0\leq\delta<1$, we have
$$
\begin{gathered}
d_{\phi^{(k)}}(a)
\leq \delta\, d_{\phi^{(k)}}\bigl((b-\epsilon)_+\bigr)\\[2pt]
a\in\operatorname{Ideal}(b)
\end{gathered}
\quad\Longrightarrow\quad
a\precsim b.
$$
We say that $(A,\phi)$ has strict $\phi$-comparison of full
elements if the same implication holds whenever $b$ is
additionally full in $M_k(A)$.
\end{definition}

\begin{lemma}\label{lem:strictphiversion}
Suppose that for all $k\in \N$, $0<\delta<1$, and positive contractions
$a,b\in M_k(A)$ we have
$$
\begin{gathered}
d_{\phi^{(k)}}(a)
\leq \delta\phi^{(k)}(b)\\[2pt]
a\in\operatorname{Ideal}(b)
\end{gathered}
\quad\Longrightarrow\quad
a\precsim b.
$$
Then $A$ has strict $\phi$-comparison. If the implication above holds whenever $b$ is additionally full, then $A$ has strict $\phi$-comparison of full elements.
\end{lemma}
\begin{proof}
Let $a,b\in M_k(A)$ be such that 
$$
d_{\phi^{(k)}}(a)\leq
\delta d_{\phi^{(k)}}((b-\epsilon)_+),
$$
where $0\leq \delta<1$ and $\epsilon>0$.
By rescaling $a$, we may assume that it is a contraction.
Let $g(t)=\min\{t/\epsilon,1\}$,
 for $t\in[0,\|b\|]$.
Then  $g(b)\sim b$, and
$$
d_{\phi^{(k)}}((b-\epsilon)_+)\leq\phi^{(k)}(g(b)).
$$
Note also that if $b$ is full, then so is $g(b)$. The assumption of the lemma gives  $a\precsim g(b)\sim b$, as desired.
\end{proof}

\section{From weak selflessness to strict $\phi$-comparison}
We fix a weakly selfless pair $(A,\phi)$ such  that $\phi$
is unital.  We set
$$
C=(*_A)_{i=1}^{\infty}\cS_\phi=C^*(A,s_1,s_2,\ldots),
$$
where the $s_i$ form a free $A$-valued semicircular family with covariance $\phi$. Let $E\colon C\to A$ denote  the canonical conditional expectation, and put
$\widetilde\phi=\phi\circ E$, regarded as a cp map on $C$.
By Lemma~\ref{lem:free-dilation}, $(C,\widetilde\phi)$ is
selfless.

\begin{lemma}\label{lem:preprocess}
Let $b\in A_+$ be a full positive contraction and $\gamma>1$.
Then there exists a positive contraction $b'\in C$ with $b'\sim b$
such that $\widetilde\phi(b')$ is invertible and
$\phi(b)\leq \gamma\widetilde\phi(b')$.
\end{lemma}

\begin{proof}
Since $b$ is full, there exist $x_1,\ldots,x_n\in A$ such that
$$
1=\sum_{i=1}^n x_i^*bx_i.
$$
Set
$$
x=\sum_{i=1}^n x_is_i\in C.
$$
Then $E(x)=0$ and
$$
E(x^*bx)=\sum_{i=1}^n\phi(x_i^*bx_i)=1.
$$

Choose $t>0$ 
such that $(1+t\|x\|)^2\leq \gamma$ and $t\|x\|<1$, and
 set 
 $$
 z=1+tx, \qquad b'=\gamma^{-1}z^*bz.
 $$
Then $z$ is invertible and  $b'$ is a positive contraction such that  
$b'\sim b$. Also, 
$$
\gamma\widetilde \phi(b')
=(\phi\circ E)(z^*bz)
=\phi(b)+t^2 1.
$$
This shows that  $\widetilde\phi(b')$ is invertible and that $\gamma\widetilde\phi(b')\geq \phi(b)$. Thus, $b'$ is as required.
\end{proof}


\begin{theorem}\label{thm:weakly-strictphi}
Let $(A,\phi)$ be weakly selfless and $\phi$ unital. Let  $0\leq \delta<1$ and let $a,b\in M_k(A)$ be positive contractions, with $b$  full.  Then
$$
d_{\phi^{(k)}}(a)
\leq \delta\phi^{(k)}(b)\Rightarrow a\precsim b.
$$
Consequently, $(A,\phi)$ has strict $\phi$-comparison of full elements.
\end{theorem}	

\begin{proof}
Since weak selflessness passes to matrix amplifications, we may assume without loss of generality that $k=1$. So let $a,b\in A$ be positive contractions such that $d_\phi(a)\leq \delta \phi(b)$, where $\delta<1$.

Choose $\gamma>1$ such that $\gamma\delta<1$. Apply the previous lemma
to obtain a positive contraction $b'\in C$ such that $\phi(b)\leq \gamma\widetilde\phi(b')$ and $\widetilde \phi(b')$ is invertible. 
Set $c=\widetilde\phi(b')$ and define  $\theta\colon C\to C$ by 
$$
\theta(x)=\widetilde\phi(c^{-\frac12}\widetilde \phi(x)c^{-\frac12})
$$
Observe that $\theta(b')=1$ and $\theta(h)\leq \gamma\delta$ for any positive contraction $h\in C^*(a)$. By Lemma \ref{lem:theta},
$\theta$ is selfless. Working in $\cT_{\theta}\supset C$, let
$\ell_\theta\in \cT_\theta$ and $s_{\theta}\in \cS_\theta$ denote the canonical left creation and semicircular elements.

Let $h\in C^*(a)$ be a positive contraction with compact support
in $(0,\|a\|]$, and choose a positive contraction $h'\in C^*(a)$ such that $h'h=h$. Then
$$
b'^{\frac12}s_{\theta} h = b'^{\frac12}\ell_{\theta} h + b'^{\frac12}\ell_{\theta}^* h.
$$
We apply the following elementary inequality (see \cite{haag}):
$$
(u+v)^*(u+v) \geq (1-t)u^*u - (t^{-1}- 1)v^*v
$$
for all $0<t<1$. Choose $u=b'^{\frac12}\ell_{\theta} h$ and $v=b'^{\frac12}\ell_{\theta}^* h$ and let us examine the terms $u^*u$ and $v^*v$. 
We have
$$
u^*u = h\ell_{\theta}^*b'\ell_{\theta}h=h\theta(b')h=h^2. 
$$
Now let $v'=b'^{\frac12}\ell_{\theta}^* h'$. Then
$$
v'(v')^*
=b'^{\frac12}\ell_\theta^*(h')^2\ell_\theta b'^{\frac12}
=b'^{\frac12}\theta((h')^2)b'^{\frac12}
\leq\delta\gamma b'\leq\delta\gamma 1.
$$
It follows that $(v')^*v'\leq\delta\gamma 1$. Multiplying by $h$ on the left
and on the right and using that $v'h=v$, we get that 
$$
v^*v= h(v')^*v'h\leq \delta\gamma h^2.
$$
We thus have that
$$
hs_\theta b' s_{\theta} h \geq (1-t)h^2 - (t^{-1}-1)\delta\gamma h^2=(1-t - (t^{-1}-1)\delta\gamma) h^2.
$$
The choice $t=\sqrt{\delta\gamma}$ yields 
$$
hs_\theta b' s_{\theta} h\geq (1-\sqrt{\delta\gamma})^2h^2.
$$
Multiplying by $a^{\frac12}$ on the left and on the right and letting $h$
range through all positive contractions in $C^*(a)$ with compact support yields
$$
a^{\frac12}s_\theta b' s_{\theta} a^{\frac12}\geq (1-\sqrt{\delta\gamma})^2a.
$$
It follows that $a\precsim  b'$ in $\cS_\theta$.
Selflessness of $\theta$ gives this comparison in an ultrapower of
$C$, and passing to coordinates, we obtain $a\precsim b'\sim b$
in $C$. Weak selflessness of $\phi$
supplies a *-homomorphism $C\to A^{\cU}$ fixing  $A$. Applying it and again
passing to coordinates we get $a\precsim b$, as desired.

The last assertion of the theorem follows from Lemma \ref{lem:strictphiversion}.
\end{proof}

Let us continue to denote by $(A,\phi)$ a weakly selfless pair
such that $\phi$ is a ucp map. In order to obtain strict $\phi$-comparison between non-full elements, we shall assume that $\phi\colon A\to A$ preserves approximate units of closed two-sided ideals of $A$. That is to say,  whenever $I\subseteq A$ is a closed two-sided ideal and 
$(e_\lambda)_\lambda$ an approximate unit of $I$, $(\phi(e_\lambda))_\lambda$ is also an approximate unit of $I$.

\begin{lemma}\label{lem:preprocess-local}
Let $(A,\phi)$ be as in the preceding paragraph. 
Let $h,b\in A$ be positive contractions such that $h\in\operatorname{Ideal}(b)$. Given $\gamma>1$ and $\epsilon>0$, there exist a positive contraction 
$b'\in C$ and $\eta>0$ 
such that $b'\sim b$ and, writing $q=\widetilde\phi(b')$, we have
$\phi(b)\leq \gamma q$ and 
\begin{equation}\label{hphih}
\|h\phi(q(q+\eta 1)^{-1})h - h^2\|\leq \epsilon.
\end{equation}
\end{lemma}

\begin{proof}
Let $I=\operatorname{Ideal}(b)$ and let
$(e_\lambda)_\lambda$ be an approximate unit of $I$
consisting of positive contractions in its Pedersen ideal.
By assumption, $(\phi^3(e_\lambda))_\lambda$
is an approximate unit of $I$. Choose $\lambda$ large
enough that
$$
\|h(1-\phi^3(e_\lambda))h\|<\frac{\epsilon}{2}.
$$
Since $e_\lambda$ is in the Pedersen ideal of $I$,
there exist $x_1,\ldots,x_n\in A$ such that
$$
e_\lambda=\sum_{i=1}^n x_i^*bx_i.
$$
Let $s_1,\ldots,s_n\in C$ be freely independent
semicircular elements with covariance $\phi$.
Set
$$
x=\sum_{i=1}^n x_is_i.
$$
Notice that $E(x)=0$ and
$E(x^*bx)=\phi(e_\lambda)$.
Choose $t>0$ such that $t\|x\|<1$ and
$(1+t\|x\|)^2<\gamma$. Now set
$$
z=1+tx,\qquad
b'=\frac{1}{\gamma}z^*bz,\qquad
\kappa=\frac{t^2}{\gamma}.
$$
Then $z$ is invertible and $b'$ is a positive contraction
such that $b'\sim b$. Also,
$$
\gamma\widetilde\phi(b')
=(\phi\circ E)(z^*bz)
=\phi(b)+t^2\phi^2(e_\lambda).
$$
Hence, $\gamma\widetilde\phi(b')\geq\phi(b)$.

Let us show that \eqref{hphih} holds for small enough
$\eta>0$. Set $q=\widetilde\phi(b')$.
Since $q\geq\kappa\phi^2(e_\lambda)$,
$0\leq\phi^2(e_\lambda)\leq1$, and
$r\mapsto r/(r+\eta)$ is operator monotone,
$$
\frac{q}{q+\eta1}
\geq
\frac{\kappa\phi^2(e_\lambda)}
{\kappa\phi^2(e_\lambda)+\eta1}
\geq
\frac{\kappa}{\kappa+\eta}\phi^2(e_\lambda).
$$
Applying $\phi$ on both sides and multiplying by $h$
on the left and on the right, we get
$$
h\phi\left(\frac{q}{q+\eta1}\right)h
\geq
\frac{\kappa}{\kappa+\eta}h\phi^3(e_\lambda)h.
$$
Since the left-hand side is dominated by $h^2$,
it follows that
$$
0\leq h^2-h\phi\left(\frac{q}{q+\eta1}\right)h
\leq h(1-\phi^3(e_\lambda))h
+\frac{\eta}{\kappa+\eta}h^2.
$$
Choosing $\eta>0$ such that
$\eta/(\kappa+\eta)<\epsilon/2$, and using the
choice of $\lambda$ and $\|h\|\leq1$, we obtain
\eqref{hphih}.
\end{proof}

\begin{theorem}\label{thm:local-covariance-comparison}
Suppose that $(A,\phi)$ is weakly selfless and that $\phi$ preserves
approximate units of ideals of $A$. Let $0\leq \delta <1$ and 
let $a,b\in M_k(A)$ be positive contractions. 
Then
$$
\begin{gathered}
d_{\phi^{(k)}}(a)
\leq \delta\phi^{(k)}(b)\\[2pt]
a\in\operatorname{Ideal}(b)
\end{gathered}
\quad\Longrightarrow\quad
a\precsim b.
$$
Consequently, $(A,\phi)$ has strict $\phi$-comparison. 
\end{theorem}

\begin{proof}
Both weak selflessness and the property of preservation of approximate units
pass from $(A,\phi)$ to $(M_k(A),\phi^{(k)})$. Let us thus assume without loss of generality that $k=1$.

Let $a,b\in A$ be positive contractions as in the theorem. 
Choose positive contractions $h,h'\in C^*(a)$
with compact support such that $h'h=h$.

We will show that $h^2\precsim b$.
Choose $\gamma>1$ and $\alpha$ such that
$\gamma\delta<\alpha<1$.

Apply Lemma~\ref{lem:preprocess-local} with $h'$
in place of $h$ and $\epsilon=1-\alpha$,
obtaining $b'\in C$ and $\eta>0$.
Set $q=\widetilde\phi(b')\in A$,
$c=(q+\eta1)^{-1/2}$, and define a cp map
$\theta\colon C\to C$ by
$$
\theta(y)=\widetilde\phi(c\widetilde\phi(y)c),
\qquad y\in C.
$$
We have
$$
0\leq\theta(b')
=\phi\left(\frac{q}{q+\eta1}\right)\leq1.
$$
The estimate \eqref{hphih} and the relation $h'h=h$
give
\begin{equation}\label{eq:theta-local-lower}
	\begin{aligned}
		h\theta(b')h
		&=h^2-h\bigl[h'(1-\theta(b'))h'\bigr]h\\
		&\geq\alpha h^2.
	\end{aligned}
\end{equation}
On the other hand, the hypothesis and conclusion of 
Lemma~\ref{lem:preprocess-local} give
$$
\phi((h')^2)\leq\delta\phi(b)
\leq\gamma\delta q.
$$
Conjugating by $c$ and applying $\widetilde\phi$,
we obtain
\begin{equation}\label{eq:theta-local-upper}
	\theta((h')^2)
	\leq\gamma\delta\theta(b')
	\leq\gamma\delta1.
\end{equation}

Since $\widetilde\phi$ is selfless and
$c\in A$ is positive,
Lemma~\ref{lem:theta} shows that $\theta$ is selfless.
Working in $\cT_\theta\supset C$, let
$\ell_\theta\in\cT_\theta$ and
$s_\theta=\ell_\theta+\ell_\theta^*\in\cS_\theta$
be the creation and semicircular operators over $C$,
and write
$$
b'^{1/2}s_\theta h
=b'^{1/2}\ell_\theta h
+b'^{1/2}\ell_\theta^*h
=:u+v.
$$
By \eqref{eq:theta-local-lower},
$$
u^*u=h\theta(b')h\geq\alpha h^2.
$$
Also, setting $v'=b'^{1/2}\ell_\theta^*h'$
and using \eqref{eq:theta-local-upper}, we have
$$
v'(v')^*
=b'^{1/2}\theta((h')^2)b'^{1/2}
\leq\gamma\delta1.
$$
Hence $(v')^*v'\leq\gamma\delta1$.
Since $v=v'h$, compressing by $h$ gives
$$
v^*v\leq\gamma\delta h^2.
$$

If $\delta>0$, we apply the inequality
$$
(u+v)^*(u+v)
\geq(1-t)u^*u-(t^{-1}-1)v^*v
$$
with
$$
t=\sqrt{\frac{\gamma\delta}{\alpha}}\in(0,1).
$$
This yields
$$
hs_\theta b's_\theta h
\geq
\bigl((1-t)\alpha-(t^{-1}-1)\gamma\delta\bigr)h^2
=(\sqrt{\alpha}-\sqrt{\gamma\delta})^2h^2.
$$
If $\delta=0$, the same estimate follows directly
from $v=0$.

Since $\alpha>\gamma\delta$, it follows that
$$
h^2\precsim hs_\theta b's_\theta h\precsim b'
$$
in $\cS_\theta$.
Selflessness of $\theta$ gives this comparison
in an ultrapower of $C$, and passing to coordinates,
we obtain $h^2\precsim b'\sim b$ in $C$.
Weak selflessness of $\phi$ supplies a
$*$-homomorphism $C\to A^{\cU}$ fixing $A$.
Applying it and passing to coordinates once more
gives $h^2\precsim b$ in $A$.
Letting $h$ range through all compactly supported
positive contractions in $C^*(a)$ yields
$a\precsim b$.

The last assertion of the theorem follows from Lemma \ref{lem:strictphiversion}.
\end{proof}

\section{From strict $\phi$-comparison to strict comparison by traces}

Here we show that if $(A,\phi)$ has strict $\phi$-comparison, then $A$ has strict comparison by traces.
We show this by proving that  $\Cu(A)$ is almost unperforated and that 
l.s.c. $[0,\infty]$-valued 2-quasitraces on $A$ are traces. We obtain
versions of these facts assuming strict $\phi$-comparison of full elements.
The latter results are somewhat less technical but follow by very similar arguments.

Let $\phi\colon A\to A$ be a ucp map.

\begin{lemma}\label{lem:rank-compression}
Let $a,b\in A$ be positive elements that 
satisfy $k[a]\leq l[b]$ in $\Cu(A)$, for $k,l\in \N$.
Then for every $\epsilon>0$ there exist positive contractions
$c\in M_{kl}(A)$ and $e\in C^*(b)$ such that
$[(a-\epsilon)_+]=[c]$, $[b]=[e]$, and
$$
d_{\phi^{(kl)}}(c)
\leq\frac{l}{k}\phi^{(kl)}(e\oplus0_{kl-1}).
$$
\end{lemma}

\begin{proof}
From $k[a]\leq l[b]$ we obtain
$\eta>0$ and $x\in M_{l\times k}(A)$ such that
$$
(a-\epsilon)_+\otimes1_k
=x^*x,\qquad xx^*\in ((b-\eta)_+\otimes1_l)M_l(A)((b-\eta)_+\otimes1_l).
$$
Write $x=(x_1,\ldots,x_k)$,
where $x_i\in M_{l\times1}(A)$ are the columns of $x$. Then
$$
x_i^*x_j=\delta_{ij}(a-\epsilon)_+.
$$
Set $d_i=x_ix_i^*$ and $e=g(b)$, where
$g(t)=\min\{t/\eta,1\}$. Thus, the $d_i$ are pairwise
orthogonal, $e\sim b$, and $(e\otimes1_l)d_i=d_i$ for all $i$.

Now let
$$
y=\frac1{\sqrt k}
\begin{pmatrix}x_1\\ \vdots\\ x_k\end{pmatrix}
\in M_{kl\times1}(A),
$$
and set $c=yy^*$. Let us show that $c,e$ are as required. 
Since $y^*y=(a-\epsilon)_+$,  $c$ is a positive
contraction such that $[c]=[(a-\epsilon)_+]$.

For each $n\geq 1$ and $i=1,\ldots,k$ we have 
$$
(c^n)_{ii}= \frac1k x_i(a-\epsilon)_+^{n-1}x_i^*=\frac1k d_i^n,
$$
where the left-hand side is the $i$-th $l\times l$ diagonal block
of $c^n$.
Taking linear combinations and passing to limits, we deduce that
for every positive contraction $f\in C_0(0,1]$, the diagonal
$l\times l$ blocks of $f(c)$ agree with $k^{-1}f(d_i)$.
Since the elements $f(d_i)$ are pairwise orthogonal positive contractions
and $(e\otimes1_l)f(d_i)=f(d_i)$, we have
$$
\sum_{i=1}^k (f(c))_{ii}
=\frac1k\sum_{i=1}^k f(d_i)
\leq\frac1k(e\otimes1_l).
$$
Consequently,
$$
\begin{aligned}
\phi^{(kl)}(f(c))
&=\frac1k\phi^{(l)}
  \left(\sum_{i=1}^k f(c)_{ii}\right)\\
&=\frac1{k^2}\phi^{(l)}
  \left(\sum_{i=1}^k f(d_i)\right)\\
&\leq\frac1{k^2}\phi^{(l)}(e\otimes1_l)=\frac1{k^2}\phi(e)
=\frac{l}{k}\phi^{(kl)}(e\oplus0_{kl-1}),
\end{aligned}
$$
for any positive contraction $f\in C_0(0,1]$.
The inequality of the lemma now follows from Remark \ref{rmk: nobidual-comparison}.
\end{proof}

\begin{theorem}\label{thm:almost-unperforated}
Let $(A,\phi)$ be a pair of a unital C*-algebra 
and a ucp map $\phi\colon A\to A$.
\begin{enumerate}[(i)]
\item 
 If $(A,\phi)$ has strict $\phi$-comparison, then $\Cu(A)$ is almost unperforated.
 \item 
If $(A,\phi)$ has strict $\phi$-comparison of full elements, then $\Cu(A)$ is almost unperforated on full elements (see Subsection~\ref{subsec:pure-cstar-algebras}).	
\end{enumerate}
\end{theorem}		

\begin{proof}
It suffices to prove almost unperforation for Cuntz classes 
of positive elements in $M_n(A)$ for some $n$, as the general case then follows
by standard approximation arguments.
Since strict $\phi$-comparison (of full elements) implies
strict $\phi^{(n)}$-comparison (of full elements) for
$(M_n(A),\phi^{(n)})$, where $\phi^{(n)}$ is viewed as
$M_n(A)$-valued by the diagonal embedding, we may assume
that $n=1$.

Suppose that $a,b\in A_+$ satisfy that
$(k+1)[a]\leq k[b]$ for some $k\in\mathbb N$.
Let $\epsilon>0$ and choose $\eta>0$ such that
$(k+1)[(a-\epsilon)_+]\leq k[(b-\eta)_+]$.
Set $N=k(k+1)$.
Apply the previous lemma to $(a-\epsilon)_+$ and
$(b-\eta)_+$, to obtain
positive contractions $c\in M_N(A)$ and
$e\in C^*((b-\eta)_+)$ such that
$[(a-2\epsilon)_+]=[c]$, $[e]=[(b-\eta)_+]$, and
$$
\begin{aligned}
d_{\phi^{(N)}}(c)
&\leq\frac{k}{k+1}\phi^{(N)}(e\oplus0_{N-1})\\
&\leq\frac{k}{k+1}
d_{\phi^{(N)}}((b-\eta)_+\oplus0_{N-1}).
\end{aligned}
$$
Note that, if $b$ is full, so is   $b\oplus0_{N-1}$. Thus, the strict $\phi$-comparison
assumption from either part (i) or (ii) gives $c\precsim b\oplus0_{N-1}$. Hence,
$$
[(a-2\epsilon)_+]=[c]\leq[b].
$$
Taking the supremum over all $\epsilon>0$, we get
$[a]\leq[b]$, as required.
\end{proof}

Let us continue to denote by $(A,\phi)$ a pair of a unital C*-algebra 
and a ucp map.

For a lower semicontinuous (l.s.c.) $[0,\infty]$-valued 2-quasitrace $\tau$ on $A$, denote by $\tau_m$ the unique 2-quasitrace on $M_m(A)$ agreeing with $\tau$ on the upper left corner (see \cite[Remark 2.27 (viii)]{blanchard-kirchberg}). We write
$$
d_\tau(c)=\lim_{n\to\infty}\tau_m(c^{1/n}),
\qquad c\in M_m(A)_+.
$$
We note that $d_\tau$ induces an order preserving functional on $\Cu(A)$
\cite{black-hand, ERS}.

\begin{lemma}\label{lem:quasitrace-order1}
Let $a,b\in M_m(A)$ be positive and let $\tau$ be a l.s.c.
$[0,\infty]$-valued $2$-quasitrace on $A$. 
\begin{enumerate}[(i)]
\item 
If $A$ has strict $\phi$-comparison and there exists a positive contraction $e\in M_m(A)$ such that
$ea=a$, $eb=b$, and  $\tau_m(e)<\infty$, then
$$
\phi^{(m)}(a)\leq\phi^{(m)}(b)
\quad\Longrightarrow\quad\tau_m(a)\leq\tau_m(b).
$$

\item 
If $A$ has strict $\phi$-comparison of full elements and $\tau$ is bounded, then
$$
\phi^{(m)}(a)\leq\phi^{(m)}(b)
\quad\Longrightarrow\quad\tau_m(a)\leq\tau_m(b).
$$
\end{enumerate}
\end{lemma}

\begin{proof}
Since strict $\phi$-comparison (of full elements) passes to
$(M_m(A),\phi^{(m)})$, replacing $\tau$ by $\tau_m$, we may
assume that $m=1$.

Let $n\geq 2$. For $i=1,\ldots,n-1$, set
$$
\xi_i(t)=\min\{1,(nt-i)_+\},
\qquad t\in[0,1],
$$
and $\xi_n=0$. Given a positive contraction $c\in A$, define
\begin{equation}\label{cprime}
c'=\bigoplus_{i=1}^n\xi_i(c)\in M_n(A).
\end{equation}
We claim that
\begin{equation}\label{sandwich1}
d_{\phi^{(n)}}(c')
\leq\phi(c),\quad \phi((c-\tfrac2n)_+)\leq d_{\phi^{(n)}}\Big((c'-\frac12)_+\Big).
\end{equation}
Indeed,
 for every positive contraction
$h\in C_0(0,1]$, we have
$$
\frac1n\sum_{i=1}^n h(\xi_i(t))\leq\frac1n\sum_{i=1}^n \mathbbm{1}_{(\frac{i}n,1]}(t)\leq t.
$$
Hence,
$$
\phi^{(n)}(h(c'))\leq\phi(c).
$$
The characterization of comparison with $d_\phi$ given in Remark \ref{rmk: nobidual-comparison} proves 
$d_{\phi^{(n)}}(c')
\leq\phi(c)$. On the other hand, from $\xi_{i+1}(t)\leq \mathbbm{1}_{(\frac{i+1}{n},1]}\leq 2(\xi_i(t)-\frac12)_+$
we deduce that 
$$
\phi(\xi_{i+1}(c))\leq d_\phi\Big((\xi_i(c)-\frac12)_+\Big)
$$
for $i=1,\ldots,n-1$. Adding over all $i$ we get
$$
\phi((c-\tfrac2n)_+)=\frac1n\sum_{i=1}^{n-1}\phi(\xi_{i+1}(c))\leq 
\frac1n\sum_{i=1}^{n-1} d_{\phi}\Big((\xi_i(c)-\frac12)_+\Big)
=d_{\phi^{(n)}}((c'-\tfrac12)_+).
$$
This proves \eqref{sandwich1}.

A similar calculation, using additivity of $\tau$ on commuting
positive elements, gives
\begin{equation}\label{sandwich2}
\tau((c-\tfrac1n)_+)
\leq\frac1n d_\tau(c')
\leq\tau(c).
\end{equation}

Now suppose that $a,b\in A_+$ satisfy $\phi(a)\leq\phi(b)$ and that
$ea=a$, $eb=b$ for a positive contraction $e\in A_+$ such that
$\tau(e)<\infty$. After a common rescaling,
assume $\|a\|,\|b\|<1$. Fix $\epsilon>0$ small enough that
$c=b+\epsilon e$ is a contraction and choose $n\in \N$ such that
$2/n<\epsilon$. Since   $e$ and $b$ commute, functional calculus in
$C^*(b,e)$ gives
\begin{equation}\label{deltanb}
(c-\tfrac2n)_+\geq(1+\epsilon-\tfrac2n)b.
\end{equation}
Indeed, on the part of the joint spectrum where $b$ is
nonzero, $e$ has value $1$, and $0\leq b\leq1$.
Set $\delta_n=(1+\epsilon-2/n)^{-1}<1$.

Form elements $a',c'$ as in  \eqref{cprime}. Then
\begin{align*}
d_{\phi^{(n)}}(a')\leq\phi(a) &\leq\phi(b)\\
&\leq\delta_n\phi((c-\tfrac2n)_+)\\
&\leq \delta_n d_{\phi^{(n)}}((c'-\tfrac12)_+),
\end{align*}
where we have used \eqref{sandwich1} and \eqref{deltanb}.
Since $\xi_1(t)\geq t-2/n$ for $t\in[0,1]$, we have
\begin{align*}
a^{1/2}\xi_1(c)a^{1/2}
&\geq
a^{1/2}\left(c-\frac2n1\right)a^{1/2}\\
&\geq
a^{1/2}\left(\epsilon e-\frac2n1\right)a^{1/2}
\\
&=
\left(\epsilon-\frac2n\right)a.
\end{align*}
Since $\epsilon>2/n$, it follows that  $a\in \operatorname{Ideal}_{A}(\xi_1(c))$, and  so 
$a'\in\operatorname{Ideal}_{M_n(A)}(c')$.
Strict $\phi$-comparison therefore gives $a'\precsim c'$.
Using \eqref{sandwich2},
$$
\tau((a-\tfrac1n)_+)
\leq\frac1n d_\tau(a')
\leq\frac1n d_\tau(c')
\leq\tau(c)=\tau(b)+\epsilon\tau(e).
$$
Let $n\to\infty$ and then $\epsilon\downarrow0$.
Since $\tau(e)<\infty$,  we get $\tau(a)\leq\tau(b)$.

In the case of strict $\phi$-comparison of full elements, we again assume that $\|a\|,\|b\|<1$, set   $e=1$ and form elements $a',c'$ as above.
Since $c'$ is full in $A$, the application of strict $\phi$-comparison is justified. Hence, as above, we deduce that $\tau(a)\leq \tau(b)$.
\end{proof}
	
\begin{theorem}\label{thm:quasitraces-are-traces}
Let $(A,\phi)$ be a pair consisting of a unital C*-algebra $A$
and a ucp map $\phi\colon A\to A$.
\begin{enumerate}[(i)]
\item
If $(A,\phi)$ has 	strict $\phi$-comparison of full elements, then every bounded 2-quasitrace on $A$ is a trace.
\item
If $(A,\phi)$ has 	strict $\phi$-comparison,  
then every lower semicontinuous $[0,\infty]$-valued
$2$-quasitrace on $A$ is a trace.
\end{enumerate}
\end{theorem}

\begin{proof}(Cf. \cite[proof of Theorem~3.6]{NgRobert}.)
Let $\tau\in QT(A)$. We aim to show that $\tau$ is additive under the hypotheses of either (i) or (ii). 

Let $a,b\in A$ be positive. Suppose first that there exists a positive contraction
$e$ such that $ae=a$, $be=b$, and $\tau(e)<\infty$. 
Then 
$$
\phi^{(2)}\Big(\begin{pmatrix}a+b&0\\0&0\end{pmatrix}\Big)=\frac12(\phi(a)+\phi(b))=\phi^{(2)}\Big(\begin{pmatrix}a&0\\0&b\end{pmatrix}\Big).
$$
It follows by the previous lemma that 
$$
\tau_2\Big(\begin{pmatrix}a+b&0\\0&0\end{pmatrix}\Big)=\tau_2\Big(\begin{pmatrix}a&0\\0&b\end{pmatrix}\Big).
$$
The left-hand side agrees with $\tau(a+b)$, while the right-hand side 
agrees with $\tau(a)+\tau(b)$. Thus, $\tau(a+b)=\tau(a)+\tau(b)$.
Assuming that $\tau$ is bounded and choosing $e=1$, we obtain part (i).

Consider now general positive $a,b\in A$. If either $\tau(a)=\infty$
or $\tau(b)=\infty$, then $\tau(a+b)=\infty$, since $\tau$ is order preserving, and so $\tau(a+b)=\tau(a)+\tau(b)$. Assume thus that $\tau(a),\tau(b)<\infty$. 
By \cite[Lemma~3.5(1)]{haag}
(see also \cite{blanchard-kirchberg}), we have
$$
\tau(a+b)^{1/2}
\leq \tau(a)^{1/2}+\tau(b)^{1/2}.
$$
Thus, $\tau(a+b)<\infty$. Let $(e_n)_n$ be an increasing sequence of positive contractions
in $C^*(a+b)$ forming an approximate unit of $C^*(a+b)$ and such that $e_{n+1}e_n=e_n$ for all $n$. Note that 
$\tau(e_n)<\infty$ for all $n$. By the result from the preceding paragraph, 
$$
\tau(e_n^{\frac12}ae_n^{\frac12})+\tau(e_n^{\frac12}be_n^{\frac12})= \tau(e_n^{\frac12}(a+b)e_n^{\frac12}),
$$
for all $n$. Hence,
$$
\tau(a^{\frac12}e_na^{\frac12})+\tau(b^{\frac12}e_nb^{\frac12})=
\tau((a+b)^{\frac12}e_n(a+b)^{\frac12}).
$$
Letting $n\to \infty$ and using the lower semicontinuity of $\tau$ we obtain $\tau(a)+\tau(b)=\tau(a+b)$. Thus $\tau$ is additive.
\end{proof}

\section{Glimm halving}

Let us continue to denote by $A$ a unital C*-algebra, and by
$\phi\colon A\to A$ a ucp map on $A$. We let
$C=C^*(A,s_1,s_2,\ldots)$,
where the $s_i$ are freely independent $A$-valued
semicircular elements with covariance $\phi$. We let
$E\colon C\to A$ be the canonical expectation and $\widetilde\phi=\phi\circ E$.

We say that a Cuntz class $s\in\Cu(A)$ is
$(2,\omega)$-divisible \cite{RR} if for every $s'\ll s$ there exist
$r\in\Cu(A)$ and $N\in\mathbb N$ such that
$$
2r\leq s,\qquad s'\leq Nr.
$$
If $s=[a]$, where $a\in(A\otimes\mathcal K)_+$, this is
equivalent to requiring that, for every
$c\in\operatorname{Ideal}(a)_+$ and $\varepsilon>0$,
there exists $y\in\overline{a(A\otimes\mathcal K)a}$ such that
$y^2=0$ and $(c-\varepsilon)_+\in\operatorname{Ideal}(y)$.

We first establish Glimm halving for full elements.

\begin{theorem}\label{thm:glimm-halving-full}
If $(A,\phi)$ is weakly selfless
and $\phi\colon A\to A$ is a ucp map,
then every full Cuntz class in $\operatorname{Cu}(A)$
is $(2,\omega)$-divisible.
\end{theorem}

\begin{proof}
Let $s\in\Cu(A)$ be full, and write
$s=\sup_j s_j$ for an increasing sequence
$(s_j)_j$ in $W(A)$.
Since $A$ is unital, fullness of $s$ implies that $s_j$
is full for all sufficiently large $j$.
As $(2,\omega)$-divisibility is preserved under increasing
suprema, it suffices to consider full classes in $W(A)$, i.e.,
$s=[a]$ with $a\in M_m(A)_+$ for some $m$.
By permanence of weak selflessness under matrix
amplifications, we may moreover assume that $m=1$.
Let us thus assume that $s=[a]$ for a full positive element
$a\in A$.

Since $a$ is full, there exist $x_1,\ldots,x_n\in A$ such that
$$
1=\sum_{i=1}^n x_i^*ax_i.
$$
Set
$$
x=\sum_{i=1}^n x_i s_i,\qquad b=x^*ax\in C.
$$
Note that $E(b)=1$.
Now put
$$
h=b^{1/2}s_{n+1}b^{1/2}.
$$
Then, using freeness at the second step,
\begin{align*}
E(h^2)
&=E(b^{1/2}s_{n+1}bs_{n+1}b^{1/2})\\
&=E(b^{1/2}E(s_{n+1}E(b)s_{n+1})b^{1/2})\\
&=E(b)=1.
\end{align*}
(Here, the free independence identity that we have used is 
$$
E(x_1y_1x_2y_2x_3) 
=E(x_1E(y_1)x_2E(y_2)x_3) + E(x_1E(y_1^{\circ}E(x_2)y_2^{\circ})x_3)
$$
where $x_1y_1x_2y_2x_3$ is an alternating product and $y_i=E(y_i)+y_i^{\circ}$.)

The C*-algebra $C$ admits an automorphism $\alpha\colon C\to C$
that fixes  $A$ and such that 
$$
\alpha(s_{n+1}) = -s_{n+1},\quad \alpha(s_i)=s_i\text { otherwise.}
$$
Namely, $\alpha$ is the Bogoliubov automorphism
induced by the operator
$T\colon\ell^2(H_\phi)\to\ell^2(H_\phi)$
that multiplies the $(n+1)$st coordinate by $-1$.
Since $\alpha(h)=-h$, we have $\alpha(h_+)=h_-$, and since
$\alpha$ preserves $E$, we obtain
$$
E(h_+^2)=E(h_-^2)=\frac12 1.
$$

Define
$$
y=h_+s_{n+2}h_-.
$$
Then $y\in\overline{bCb}$ and $y^2=0$.
Let us show that $y$ is full.
Using again freeness,
$$
\begin{aligned}
E(yy^*)
&=E\bigl(h_+s_{n+2}h_-^2s_{n+2}h_+\bigr)\\
&=E\bigl(h_+E(s_{n+2}E(h_-^2)s_{n+2})h_+\bigr)\\
&=\frac12 E(h_+^2)
=\frac14 1.
\end{aligned}
$$
Consequently,
$$
\widetilde\phi(yy^*)=\frac14 1.
$$
Since $\widetilde\phi$
is selfless by Lemma~\ref{lem:free-dilation},
it preserves closed two-sided ideals of $C$
\cite[Theorem~5.3]{GJKEPR}. Thus, $y$ is full in $C$.

We have
$$
2[y^*y]\leq[b]\leq[a],
$$
and $[y^*y]$ is full in $\Cu(C)$.
Thus, $[a]\in\Cu(C)$ is $(2,\omega)$-divisible.
By weak selflessness, there is a *-homomorphism
$\sigma\colon C\to A^{\cU}$ fixing $A$.
Therefore,
$[a]\in\Cu(A)\subseteq\Cu(A^{\cU})$
is $(2,\omega)$-divisible in $\Cu(A^{\cU})$.
By an ultrapower descent argument,
we obtain that $[a]$ is $(2,\omega)$-divisible in
$\Cu(A)$, as required; see
\cite[proof of Corollary~5.3]{VilaltaGlimmSurvey}.
\end{proof}

The C*-algebra $A$ is said to satisfy the global Glimm halving
property if all $s\in\Cu(A)$ are
$(2,\omega)$-divisible.

\begin{lemma}\label{lem:vacuum-ideal-preserving}
Suppose that $\phi$ preserves approximate units of ideals of $A$.
Then the expectation $E\colon C\to A$ preserves ideals of $C$,
i.e., $E(J)=J\cap A$ for every closed two-sided ideal $J$ of $C$.
\end{lemma}

\begin{proof}
By Lemma~\ref{lem:free-dilation}, the map
$\widetilde\phi=\phi\circ E\colon C\to C$ is selfless.
In particular, it is approximately inner
\cite[Theorem~5.3]{GJKEPR}, and hence preserves ideals.

Let $J$ be an ideal of $C$. Let $c\in J$ and set $K=\operatorname{Ideal}_A(E(c))$.
For every $x,y\in A$, we have
$$
\phi(xE(c)y)
=\widetilde\phi(xcy)
\in J\cap A.
$$
It follows that $\phi(K)\subseteq J\cap A$.
If $(e_\lambda)_\lambda$ is an approximate unit of $K$,
then $(\phi(e_\lambda))_\lambda\subset J\cap A$ is also an approximate
unit of $K$. Hence $K\subseteq J\cap A$, and in particular $E(c)\in J$.
This proves $E(J)\subseteq J\cap A$.
The reverse inclusion follows because $E$ fixes $A$.
\end{proof}

\begin{theorem}\label{thm:glimm-halving}
Let $A$ be a unital C*-algebra, and let
$\phi\colon A\to A$ be a weakly selfless ucp map
preserving approximate units of ideals of $A$.
Then $A$ has the global Glimm halving property.
\end{theorem}

\begin{proof}
By finite-matrix approximation and passage to increasing
suprema, it suffices to prove $(2,\omega)$-divisibility
for classes represented by positive elements in matrix
algebras over $A$.
Both weak selflessness and preservation of approximate
units of ideals pass to matrix amplifications.
We may therefore fix a positive $a\in A$ and show that $[a]$ is
$(2,\omega)$-divisible.

Set $I=\operatorname{Ideal}_A(a)$, and choose an approximate
unit $(e_\lambda)_\lambda$ of $I$ consisting of positive
contractions in its Pedersen ideal.
Define
$$
f_\lambda=\phi(e_\lambda),\qquad
g_\lambda=f_\lambda\phi(f_\lambda^2)f_\lambda,\qquad
k_\lambda=g_\lambda\phi(g_\lambda)g_\lambda.
$$
Using that $\phi$ preserves approximate units of ideals,
we readily deduce that
$(f_\lambda)_\lambda$, $(g_\lambda)_\lambda$, and
$(k_\lambda)_\lambda$ are approximate units of $I$.

Fix $\varepsilon>0$. Choose $\lambda$ such that
$$
(a-\varepsilon/2)_+
\in\operatorname{Ideal}_A(k_\lambda).
$$

Since $e_\lambda$ belongs to the Pedersen ideal of $I$, we have
$$
e_\lambda=\sum_{i=1}^n x_i^*ax_i
$$
with $x_i\in A$.
Set
$$
x=\sum_{i=1}^n x_i s_i,\qquad b=x^*ax.
$$
Then
$$
E(b)=\sum_{i=1}^n\phi(x_i^*ax_i)=\phi(e_\lambda)=f_\lambda.
$$

Now set
$$
h=bs_{n+1}b.
$$
By freeness and Schwarz's inequality,
$$
\begin{aligned}
E(h^2)
&=E(b\phi(E(b^2))b)\\
&\geq f_\lambda\,\phi(E(b^2))\,f_\lambda\\
&\geq f_\lambda\,\phi(f_\lambda^2)\,f_\lambda
=g_\lambda.
\end{aligned}
$$
The Bogoliubov automorphism $\alpha\colon C\to C$
fixing $s_i$ for all $i\neq n+1$ and mapping
$s_{n+1}$ to $-s_{n+1}$ fixes $b$ and preserves $E$.
Hence,
$$
p:=E(h_+^2)=E(h_-^2)=\frac12E(h^2)\geq\frac12g_\lambda.
$$

Define
$$
y=h_+^2s_{n+2}h_-.
$$
Then $y\in\overline{bCb}$ and $y^2=0$.
A further application of freeness and Schwarz's inequality gives
$$
\begin{aligned}
E(yy^*)
&=E\bigl(h_+^2\,\phi(E(h_-^2))\,h_+^2\bigr)\\
&=E\bigl(h_+^2\,\phi(p)\,h_+^2\bigr)\\
&\geq p\phi(p)p.
\end{aligned}
$$

Let $J=\operatorname{Ideal}_C(y)$.
By Lemma~\ref{lem:vacuum-ideal-preserving},
$E(yy^*)\in J$, and hence $p\phi(p)p\in J$. That is,
the images of $p$ and $\phi(p)$ in $C/J$ are orthogonal.
Since
$g_\lambda\leq2p$ and  $\phi(g_\lambda)\leq2\phi(p)$,
the images of $g_\lambda$ and $\phi(g_\lambda)$ in $C/J$ are also orthogonal.
Thus,
$$
k_\lambda=g_\lambda\phi(g_\lambda)g_\lambda\in J.
$$
Our choice of $\lambda$ now gives
$(a-\varepsilon/2)_+\in\operatorname{Ideal}_C(y)$.

Since $y^2=0$ and $y\in\overline{bCb}$, we have
$$
2[y^*y]=[y^*y+yy^*]\leq[b]\leq[a]
$$
in $\Cu(C)$.
On the other hand, from $(a-\varepsilon/2)_+\in\operatorname{Ideal}_C(y)$ 
we deduce that
$$
[(a-\varepsilon)_+]
\ll [(a-\varepsilon/2)_+]
\leq \infty [y^*y].
$$
Thus, for some $N\in\mathbb N$,
$[(a-\varepsilon)_+]\leq N[y^*y]$.
Since $\varepsilon>0$ was arbitrary, we conclude that $[a]$ is
$(2,\omega)$-divisible in $\Cu(C)$.

By weak selflessness, there is a *-homomorphism
$\sigma\colon C\to A^{\mathcal U}$ fixing $A$.
Consequently, $[a]$ is
$(2,\omega)$-divisible in $\Cu(A^{\cU})$.
As in the proof of Theorem~\ref{thm:glimm-halving-full},
this in turn implies that $[a]$ is $(2,\omega)$-divisible
in $\Cu(A)$, as required.
\end{proof}

\begin{proof}[Proof of Theorem \ref{thm:main2}]
Suppose that $(A,\phi)$ is weakly selfless and that $\phi$
preserves approximate units of ideals of $A$. Then $(A,\phi)$
has strict $\phi$-comparison, by Theorem \ref{thm:local-covariance-comparison}. By Theorem \ref{thm:almost-unperforated}(i),   $\Cu(A)$  is almost unperforated, equivalently, $A$ has strict comparison by $2$-quasitraces. By Theorem~\ref{thm:quasitraces-are-traces}(ii),
the l.s.c. $[0,\infty]$-valued $2$-quasitraces on $A$ are traces.
Hence, $A$ has strict comparison by traces by Proposition \ref{prop:comp-by-traces}. By Theorem \ref{thm:glimm-halving}, $A$ has the global Glimm halving property. Thiel's theorem \cite[Theorem 4.4]{thiel-ranks} now implies that $\Cu(A)$ is almost divisible. Therefore, $A$ is pure. 
\end{proof}

\begin{lemma}\label{lem:expectation-approximate-units}
A conditional expectation $E\colon A\to B$ preserves approximate units 
of ideals of $A$ if and only if $I=\overline{(I\cap B)A}$
 for every closed two-sided ideal $I\subseteq A$.
\end{lemma}

\emph{Note}: In the terminology used in \cite{KwasniewskiMeyer}, the condition in the lemma says that the inclusion $B\subseteq A$ is symmetric and that $B$ separates ideals in $A$.

\begin{proof}
Let $I$ be an ideal of $A$ and let $(e_\lambda)_\lambda$ be an approximate unit of $I$. Set $b_\lambda=E(e_\lambda)$. 

Suppose that $(b_\lambda)_\lambda$ is an approximate unit of $I$. Then 
for every $a\in I$ we have 
$$
a=\lim_\lambda b_\lambda a\in \overline{(I\cap B)A}.
$$
This proves one implication.

Suppose that $I=\overline{(I\cap B)A}$. If $b\in I\cap B$, then 
$$
b=E(b)=\lim_\lambda E(e_\lambda b)=\lim_\lambda b_\lambda b.
$$
Thus, $(b_\lambda)_\lambda$ is an approximate unit of $I\cap B$. Since
$I=\overline{(I\cap B)A}$, this extends to $I$.
\end{proof}

\begin{proof}[Proof of Corollary \ref{cor:tensor}]
In both cases, the assumptions guarantee that every ideal of  $A\otimes B$
has the form $A\otimes J$, where $J\subseteq B$ is an ideal of $B$; 
see \cite{blanchard-kirchberg}. Hence, ideals of $A\otimes B$ satisfy
that $A\otimes J=\overline{(1\otimes J)(A\otimes B)}$. 
By Lemma~\ref{lem:expectation-approximate-units}, any conditional expectation $A\otimes B\to  1\otimes B$ preserves approximate units of ideals. Applying this to the expectation
$a\otimes b\mapsto 1_A\otimes\rho(a)b$,
which is selfless by
\cite[Theorem~5.4(iv) and Remark~5.5]{GJKEPR},
we obtain the corollary.
\end{proof}

\end{document}